\documentclass[a4paper,twopage,reqno,12pt]{amsart}
\usepackage[top=30mm,right=30mm,bottom=30mm,left=30mm]{geometry}

\usepackage{graphicx}
\usepackage{amsmath}
\usepackage{amssymb}
\usepackage{amsfonts}
\usepackage{amsthm}
\usepackage{amstext}
\usepackage{amsbsy}
\usepackage{amsopn}
\usepackage{amscd}
\usepackage{enumerate}

\newtheorem{theorem}{Theorem}[section]
\newtheorem{corollary}[theorem]{Corollary}
\newtheorem{lemma}[theorem]{Lemma}
\newtheorem{proposition}[theorem]{Proposition}

\theoremstyle{definition}

\numberwithin{equation}{section}

\newcommand{\m}{\textsf{m}}
\newcommand{\n}{\textsf{n}}
\newcommand{\nse}{\textsf{nse}}
\newcommand{\Aut}{\textsf{Aut}}
\newcommand{\GF}{\textsf{GF}}
\newcommand{\F}{\mathbb{F}}
\newcommand{\Out}{\textsf{Out}}
\newcommand{\N}{\textbf{N}}
\newcommand{\C}{\textbf{C}}

\newcommand{\Sz}{\mathrm{Sz}}
\renewcommand{\varphi}{\phi}

\begin{document}

\title[A characterization of Suzuki groups]{Finite groups of the same type as Suzuki groups}

\author[S.H. Alavi]{Seyed Hassan Alavi}
\address{S.H. Alavi, Department of Mathematics, Faculty of Science, Bu-Ali Sina University, Hamedan, Iran}
\email{alavi.s.hassan@gmail.com (preferred)}
\email{alavi.s.hassan@basu.ac.ir}

\author[A. Daneshkhah]{Ashraf Daneshkhah$^*$}
\thanks{Corresponding author: A. Daneshkhah}
\address{A. Daneshkhah, Department of Mathematics, Faculty of Science, Bu-Ali Sina University, Hamedan, Iran}
\email{daneshkhah.ashraf@gmail.com (preferred)}
\email{adanesh@basu.ac.ir}

\author[H. Parvizi Mosaed]{Hosein Parvizi Mosaed}
\address{H. Parvizi Mosaed, Alvand Institute of Higher Education, Hamedan, Iran}
\email{h.parvizi.mosaed@gmail.com}

\subjclass[2010]{Primary 20D60; Secondary 20D06.}
\keywords{Thompson's problem, Element order, Suzuki group}

\maketitle%

\begin{abstract}
  For a finite group $G$ and a positive integer $n$, let $G(n)$ be the set of all elements in $G$ such that $x^{n}=1$. The groups $G$ and $H$ are said to be of the same (order) type if $G(n)=H(n)$, for all $n$. The main aim of this paper is to show that if $G$ is a finite group of the same type as Suzuki groups $\Sz(q)$, where $q=2^{2m+1}\geq 8$, then $G$ is isomorphic to $\Sz(q)$. This addresses the well-known J. G. Thompson's problem (1987) for simple groups.
\end{abstract}

\section{Introduction}

For a finite group $G$ and a positive integer $n$, let $G(n)$ consist of all elements $x$ satisfying $x^{n} = 1$. The type of $G$ is defined to be the function whose value at $n$ is the order of $G(n)$. In 1987, J. G. Thompson \cite[Problem 12.37]{book:khukh} possed a problem whether it is true that a group is solvable if its type is the same as that of a solvable one? 
This problem links to the set $\nse(G)$ of \emph{the number of elements of the same order} in $G$. Indeed, it turns out that if two groups $G$ and $H$ are of the same type, then $\nse(G)=\nse(H)$ and $|G|=|H|$. Therefore, if a group $G$ has been uniquely determined by its order and $\nse(G)$, then Thompson's problem is true for $G$. One may ask this problem for non-solvable groups, in particular, finite simple groups. In this direction, Shao et al \cite{art:Shao} studied finite simple groups whose order is divisible by at most four primes. Following this investigation, such problem has been studied for some families of simple groups \cite{art:ADP-Sz} including small Ree groups. In this paper, we prove that
\begin{theorem}\label{thm:main}
Let $G$ be a group with $\nse(G)=\nse(\Sz(q))$ and $|G|=|\Sz(q)|$. Then $G$ is isomorphic to $\Sz(q)$.
\end{theorem}

As noted above, as an immediate consequence of Theorem~\ref{thm:main}, we have that

\begin{corollary}
If $G$ is a finite group of the same type as $\Sz(q)$, then $G$ is isomorphic to $\Sz(q)$.
\end{corollary}

In order to prove Theorem~\ref{thm:main}, we use a partition of Suzuki groups $S:=\Sz(q)$, where $q=2^{2m+1}\geq 8$ (see Lemma~\ref{lem:part}), that is to say, a set of subgroups $H_{i}$ of $S$, for $i=1,\ldots,s$, such that each nontrivial element of $S$ belongs to exactly one subgroup $H_{i}$. We use this information to determine the set $\nse(S)$ in Proposition~\ref{prop:nse} and to prove that $2$ is an isolated vertex in the prime graph of a group $G$ satisfying hypotheses of Theorem~\ref{thm:main}, see Proposition~\ref{prop:isolated}. Then we show that $G$ is neither Frobenius, nor $2$-Frobenius group. Finally, we obtain a section of $G$ which is isomorphic to $S$ and  prove that $G$ is isomorphic to $S$.

Finally, some brief comments on the notation used in this paper. Throughout this article all groups are finite. Our group-theoretic notation is standard, and it is consistent with the notation in \cite{book:Car,book:atlas,book:Gor}. We denote a Sylow $p$-subgroup of $G$ by $G_p$. We also use $\n_p(G)$ to  denote the number of Sylow $p$-subgroups of $G$. For a positive integer $n$, the set of prime divisors of $n$ is denoted by $\pi(n)$, and if $G$ is a finite group, $\pi(G):=\pi(|G|)$, where $|G|$ is the order of $G$. We denote the set of elements' orders of $G$ by $\omega(G)$ known as \emph{spectrum} of $G$. The \emph{prime graph} $\Gamma(G)$ of a finite group $G$ is a graph whose vertex set is $\pi(G)$, and two vertices $p$ and $q$ are adjacent if and only if $pq\in\omega(G)$. Assume further that $\Gamma(G)$ has $t(G)$ connected components $\pi_i$, for $i=1,2,\hdots,t(G)$. The positive integers $n_{i}$ with $\pi(n_{i})=\pi_{i}$ are called order components of $G$. In the case where $G$ is of even order, we always assume that $2\in\pi_1$, and $\pi_{1}$ is said to be the even component of $G$. In this way, $\pi_{i}$ and $n_{i}$ are called odd components and odd order components of $G$, respectively. Recall that $\nse(G)$ is the set of the number of elements in $G$ with the same order. In other word, $\nse(G)$ consists of the numbers $\m_i(G)$ of elements of order $i$ in $G$, for $i\in \omega(G)$. Here,  $\varphi$ is the \emph{Euler totient} function.

\section{Preliminaries}\label{sec:}

In this section, we introduce the some known results which will be used in the proof of the main result.

\begin{lemma}\cite[Theorem 9.1.2]{book:Hall}\label{27}
Let $G$ be a finite group, and let $n$ be a positive integer dividing $|G|$. Then $n$ divides $|G(n)|$.
\end{lemma}

The proof of the following result is straightforward by Lemma \ref{27}. Recall that $\nse(G)=\{\m_{i}(G) \mid i\in \omega(G)\}$.

\begin{lemma}\label{28}
Let $G$ be a finite group. Then for every $i\in\omega(G)$, $\varphi(i)$  divides $\m_i(G)$, and $i$ divides $\sum_{j \mid i} \m_j(G)$. Moreover, if $i>2$, then $\m_i(G)$ is even.
\end{lemma}

\begin{lemma}[Theorem 3 in \cite{art:weisner}]\label{20}
Let $G$ be a finite group of order $n$. Then the number of elements whose orders are multiples of $t$ is either zero, or a multiple of the greatest divisor of $n$ that is prime to $t$.
\end{lemma}

In what follows, recall that $t(G)$ is the number of connected components
of the prime graph $\Gamma(G)$.

\begin{lemma}\cite[Theorem 1]{art:Chen}\label{lem:frob}
Let $G$ be a Frobenius group of even order with kernel $K$ and complement $H$. Then $t(G)=2$, $\pi(H)$ and $\pi(K)$ are vertex sets of the connected components of $\Gamma(G)$.
\end{lemma}

A group $G$ is called $2$-Frobenius if there exists a normal series $1\unlhd H\unlhd K\unlhd G$ such that $G/H$ and $K$ are Frobenius groups with kernel $K/H$ and $H$ respectively.

\begin{lemma}\cite[Theorem 2]{art:Chen}\label{lem:2-frob}
Let $G$ be a $2$-Frobenius group of even order. Then $t(G)=2$, $\pi(G/K)\cup\pi(H)=\pi_1$,  $\pi(K/H)=\pi_2$, and $G/K$ and $K/H$ are cyclic groups and $|G/K|$ divides $|\Aut(K/H)|$.
\end{lemma}

\section{Elements of the same order in Suzuki groups}

In this section, we determine the set of the number of elements of the same order in Suzuki groups.

\begin{lemma}[\cite{art:Shi}] \label{lem:omega} Let $S= \Sz(q)$ with $q=2^{2m+1}\geq 8$. Then $\omega(S)$ consists of all factors of $4$, $q-1$ and $q\pm\sqrt{2q}+1$.
\end{lemma}

Let $G$ be a group, and let $H_1,\hdots, H_t$ be subgroups of $G$. Then the set $\{H_1,\hdots, H_t\}$ forms a partition of $G$ if each non-trivial element of $G$ belongs to exactly one subgroup $H_i$ of $G$. Lemma~\ref{lem:part} below introduces a partition of Suzuki groups.

\begin{lemma}\label{lem:part}
Let $S=\Sz(q)$ with $q=2^{2m+1}\geq 8$, and let $\F:=\GF(q)$. Then
\begin{enumerate}[{\quad \rm (a)}]
\item $S$ possesses cyclic subgroups $U_1$ and $U_2$ of orders $q +\sqrt{2q }+1$ and $q-\sqrt{2q}+1$, respectively;
\item if $1 \neq u\in U_i$, for $i=1,2$, then $\C_S(u)=U_i$. Moreover, $|\N_S(U_i):U_i|=4$;
\item $S$ possesses a cyclic subgroup $V$ of order $q-1$ and $|\N_S(V):V|=2$;
\item $S$ possesses a $2$-subgroup $W$ of orders $q^2$ and exponent $4$ and $|S:\N_S(W)|=q^{2}+1$; Moreover, the elements of $W$ are of the form
    \begin{align}\label{eq:w}
      w(a,b):=\left(
        \begin{array}{cccc}
          1 & 0 & 0 & 0 \\
          a & 1 & 0 & 0 \\
          b & a\pi & 1 & 0 \\
          a^{2}(a\pi)+ab+b\pi & a(a\pi)+b & a & 1 \\
        \end{array}
      \right),
    \end{align}
    where $a,b \in \F$ and $\pi\in \Aut(\F)$ maps $x$ to $x^{2^{m+1}}$, for all $x\in \F$.
\item the conjugates of $U_1$, $U_2$, $V$ and $W$ form a partition of $S$.
\end{enumerate}
\end{lemma}
\begin{proof}
All parts of this result follow from Lemma 3.1 and Theorem 3.10 in \cite{book:Hup} except for the facts that $|\N_S(V):V|=2$ and $|S:\N_S(W)|=q^{2}+1$ which can be found in the proof of Theorem 3.10 in \cite{book:Hup}.
\end{proof}


\begin{proposition}\label{prop:nse}
Let $S=\Sz(q)$ with $q=2^{2m+1}\geq 8$. Then the set $\nse(S)$ consists of exactly one the following numbers
\begin{enumerate}[{ \quad \rm (a)}]
  \item $1$, $(q-1)(q^2+1)$, $q(q-1)(q^2+1)$;
  \item $\varphi(i)q^2(q\mp\sqrt{2q}+1)(q-1)/4$, where $i>1$ divides $q\pm\sqrt{2q}+1$;
  \item $\varphi(i)q^2(q^2+1)/2$, where $i>1$ divides $q-1$.
\end{enumerate}
\end{proposition}
\begin{proof}
Suppose $i\in \omega(S)$ is an even number. Then by Lemma~\ref{lem:omega}, we have that $i=2$ or $i=4$ and $i$ divides the order of subgroup $W$ as in Lemma~\ref{lem:part}(d). Then each element of $W$ is of the form $w(a,b)$ as in \eqref{eq:w}. Obviously,
\begin{align*}
  w(a,b)w(c,d)=w(a+c,b+d+(a\pi)c).
\end{align*}
This in particular shows that $w(0,b)$ (with $b\neq 0$) are the only elements of $W$ of order $2$. Therefore, the number of involutions in $W$ is $q-1$. Since $W$ is a part of the partition introduced in Lemma~\ref{lem:part}(e), the elements of order $2$ of $S$ belong to exactly one of the conjugates of $W$. Thus by Lemma~\ref{lem:part}(d), there are $q^{2}+1$ conjugates of $W$ implying that there are exactly $\m_{2}(S)=(q-1)(q^{2}+1)$ involutions in $S$. If also follows from Lemma~\ref{lem:part}(d) that the number of elements of order $4$ in $W$ is $q^{2}-q$, and hence applying the partition in Lemma~\ref{lem:part}(e), we conclude that $S$ consists of $\m_{4}(S)=q(q-1)(q^{2}+1)$ elements of order $4$. This proves part (a).

Suppose now $i\in \omega(S)$ is an odd number. Then, by Lemma~\ref{lem:omega}, $i$ divides the order of one the cyclic subgroups $U_{1}$, $U_{2}$ and $V$ as in Lemma~\ref{lem:part}, say $H$. Assume $i=np^{\alpha}$ with $p$ odd. Since $H$ is a part of the partition introduced in Lemma~\ref{lem:part}(e), the elements of order $i$ are contained in $H$ and its conjugates. Since also $H$ is cyclic, there are $\varphi(i)$ elements of order $i$ in each conjugates of $H$ including $H$. Note that each conjugate of $H$ contains exactly one Sylow $p$-subgroup of $S$. Then there are $\n_{p}(S)$ conjugates of $H$ in $S$. Therefore, the number $\m_{i}(S)$ of elements of order $i$ is $\varphi(i)\n_{p}(S)$.

Now we consider each possibility of $H$. If $H=U_{t}$ of order $q\pm\sqrt{2q}+1$, for $t=1,2$, then by Lemma~\ref{lem:part}(b),  $|\N_S(U_t):U_t|=4$, and so $|S:\N_S(U_t)|=|S|/4|U_t|$, for $t=1,2$. Since all conjugates of $U_1$, $U_2$, $V$ and $W$ as in Lemma~\ref{lem:part} form a partition of $S$, it follows that $\n_p(S)=|S:\N_S(S_p)|=|S:\N_S(U_t)|=|S|/4|U_{t}|$. Therefore, $\m_i(S)=\varphi(i)\n_{p}(S)=q^2(q-1)(q\mp\sqrt{2q}+1)/4$. This follows part (b). If $H=V$, then Lemma~\ref{lem:part}(c) implies that $|\N_S(V):V|=2$, and so the same argument as in the previous cases, we conclude that $\n_p(S)=|S|/2|V|=q^2(q^{2}+1)/2$. This follows (c).
\end{proof}

\section{Proof of the Main Theorem}
In this section, we prove Theorem \ref{thm:main}. From now on, set $S:=\Sz(q)$, where $q=2^{2m+1}\geq 8$, and recall that $G$ is a finite group with $\nse(G)=\nse(S)$ and $|G|=|S|$. Therefore, by Proposition~\ref{prop:nse}, $\nse(S)$ consists of
\begin{align}\label{eq:nse}
\nonumber  \m_{1}(S)=&1; \\
\nonumber  \m_{2}(S)=&(q-1)(q^2+1); \\
  \m_{4}(S)=&q(q-1)(q^2+1); \\
\nonumber  \m_{i}(S)=&\varphi(i)q^2(q\mp\sqrt{2q}+1)(q-1)/4, \text{ where $i>1$ divides $q\pm\sqrt{2q}+1$};\\
\nonumber  \m_{i}(S)=&\varphi(i)q^2(q^2+1)/2, \text{ where $i>1$ divides $q-1$}.
\end{align}

\begin{proposition}\label{prop:isolated}
The vertex $2$ is an isolated vertex in $\Gamma(G)$.
\end{proposition}
\begin{proof}
Assume the contrary. Then there is an odd prime divisor $p$ of $|G|$ such that $2p\in\omega(G)$. Let $f(n)$ be the number of elements of $G$ whose orders are multiples of $n$. Then by Lemma \ref{20}, $f(2)$ is a multiple of the greatest divisor of $|G|$ that is prime to $2$. Since $(q^2+1)(q-1)$ is the greatest divisor of $|G|$ which is  coprime to $2$, there exists a positive integer $r$ such that $f(2)=(q^2+1)(q-1)r$ and $(r,2)=1$. On the other hand, by Lemma \ref{28}, it is obvious that $\m_2(G)=\m_2(S)$, and so
\begin{align*}
  f(2)= & \m_2(G)+\sum_{i>2 \ is\ even}\m_i(G),
\end{align*}
with $\m_{i}$ as in \eqref{eq:nse}. Now applying Proposition~\ref{prop:nse}, there is a non-negative integer $\alpha$ such that
\begin{align*}
  f(2)= &(q^2+1)(q-1)+\alpha q(q^2+1)(q-1)+g(2),
\end{align*}
where
\begin{align*}
  g(2)= \sum_{\substack{i \mid q\pm\sqrt{2q}+1 \\ i\neq 1}} \beta_{i}\cdot  \m_{i}(S) + \sum_{\substack{i\mid q-1 \\ i\neq 1}} \gamma_{i}\cdot \m_{i}(S)
\end{align*}
for some non-negative integers $\beta_{i}$ and $\gamma_{i}$. Since $2p\in\omega(G)$, we have that $\alpha q(q^2+1)(q-1)+g(2)>0$. Then
\begin{align*}
  g(2)=(q^2+1)(q-1)(r-1-\alpha q).
\end{align*}

We now prove that $q^2$ divides $g(2)$. It follows from Lemma \ref{lem:omega} that $2$ is an isolated vertex of $\Gamma(S)$. Then a Sylow $2$-subgroup of $S$, say $S_2$, acts fixed point freely (by conjugation) on the set of elements of order $i\neq 1,2,4$ (see Proposition~\ref{prop:nse} and \eqref{eq:nse}). Thus $|S_2|$ divides $\m_i(S)$ with $i\neq 1,2,4$. Hence $q^2$ divides $\m_i(S)$ implying that $g(2)$ is a multiple of $q^{2}$.

We now consider the following two cases:\smallskip

\noindent \textbf{(1)}  Let $g(2)\neq 0$. Then $q^2$ divides $r-1-\alpha q$, and so $q^2+\alpha q+1 \leq r$. This implies that $|G|=q^2(q^2+1)(q-1)<(q^2+1)(q-1)r=f(2)$, which is impossible.\smallskip

\noindent \textbf{(2)} Let $g(2)=0$. Then $r-1-\alpha q=0$ and $\alpha\neq 0$, and so $\m_{2p}(G)=q(q^2+1)(q-1)$. Therefore
\begin{align*}
f(p)=&\sum_{p\mid i}\m_i(G)=\alpha' q(q^2+1)(q-1)+\sum_{\substack{i \mid q\pm\sqrt{2q}+1 \\ i\neq 1}} \beta_{i}'\cdot  \m_{i}(S) + \sum_{\substack{i\mid q-1 \\ i\neq 1}} \gamma_{i}'\cdot \m_{i}(S),
\end{align*}
where $\alpha'$, $\beta_{i}'$ and $\gamma_{i}'$  are non-negative integers. Since $\m_{2p}=q(q^2+1)(q-1)$, we have that $\alpha'>0$. On the other hand, by Lemma \ref{20}, $f(p)=q^2(q^2+1)(q-1)r'/|G_p|$ with $r'$ a positive integer. Thus
\begin{align}\label{eq:iso-1}
  \frac{q^2(q^2+1)(q-1)r'}{|G_{p}|}=\alpha' q(q^2+1)(q-1)+\sum_{\substack{i \mid q\pm\sqrt{2q}+1\\ i\neq 1}} \beta_{i}'\cdot  \m_{i}(S) + \sum_{\substack{i\mid q-1 \\ i\neq 1}} \gamma_{i}'\cdot \m_{i}(S).
\end{align}
Since $q^2$ divides both $q^2(q^2+1)(q-1)r'/|G_p|$ and $\m_i(S)$ in \eqref{eq:iso-1}, it follows that $q^2$ divides $\alpha' q(q^2+1)(q-1)$. Then $q\mid \alpha'$, and so $|G|=q^2(q^2+1)(q-1)\leq \alpha' q(q^2+1)(q-1)\leq f(p)$, which is impossible.
\end{proof}

\begin{proposition}\label{32}
The group $G$ has a normal series $1\unlhd H \unlhd K \unlhd G$ such that $H$ and $G/K$ are $\pi_1$-groups and $K/H$ is a non-abelian simple group, $H$ is a nilpotent group and $|G/K|$ divides $|\Out(K/H)|$.
\end{proposition}
\begin{proof}
By Proposition~\ref{prop:isolated}, the vertex $2$ is an isolated vertex in the prime graph $\Gamma(G)$ of $G$. This implies that the number $t(G)$ of connected components
of the prime graph $\Gamma(G)$ is at least two. The assertion follows from \cite[Theorem A]{art:Williams} provided that $G$ is neither a Frobenius group, nor a $2$-Frobenius group.

Let $G$ be a Frobenius group with kernel $K$ and complement $H$.  Then by Lemma \ref{lem:frob}, we must have $t(G)=2$, $\pi(H)$ and $\pi(K)$ are vertex sets of the connected components of $\Gamma(G)$. By Proposition~\ref{prop:isolated}, the vertex $2$ is an isolated vertex in $\Gamma(G)$. Then either (i) $|K|=q^2$ and $|H|=(q^2+1)(q-1)$, or (ii) $|H|=q^2$ and $|K|=(q^2+1)(q-1)$. Both cases can be ruled out as $|H|$ must divide $|K|-1$.

Let $G$ be a $2$-Frobenius group. Then Lemma \ref{lem:2-frob} implies that $t(G)=2$ and $G$ has a normal series $1\unlhd H\unlhd K\unlhd G$ such that $G/H$ and $K$ are Frobenius groups with kernel $K/H$ and $H$ respectively, $\pi(G/K)\cup\pi(H)=\pi_1$, $\pi(K/H)=\pi_2$ and $|G/K|$ divides $|{\rm \Aut}(K/H)|$. Since $2$ is an isolated vertex of $\Gamma(G)$ by Proposition~\ref{prop:isolated}, $|K/H|=(q^2+1)(q-1)$ and $|G/K|.|H|=q^2$. Since also $K$ is a Frobenius group with kernel $H$, there is a positive integer $\alpha$ such that $(q^2+1)(q-1)$ divides $2^\alpha-1$, which is a contradiction.
\end{proof}

\subsection{Proof of Theorem \ref{thm:main}}
\begin{proof}
Let $S:=\Sz(q)$, where $q=2^{2m+1}\geq 8$. Suppose that $G$ is a finite group with $\nse(G)=\nse(S)$ and $|G|=|S|$. By applying Proposition~\ref{32}, the group $G$ has a normal series $1\unlhd H \unlhd K \unlhd G$ such that $H$ and $G/K$ are $\pi_1$-groups and $K/H$ is a non-abelian simple group. Since $3$ is a prime divisor of all finite non-abelian simple groups except for Suzuki groups. Moreover, $3$ is coprime to $|K/H|$. Then $K/H\cong \Sz(q')$, where $q'=2^{2m'+1}$. This, in particular, implies that $2^{4m'+2}$ divides $2^{4m+2}$, and hence $m'\leq m$. On the other hand, $H$ and $G/K$ are $\pi_1$-groups. Then $(q^2+1)(q-1)$ divides $|K/H|$, and so $(q^2+1)(q-1)$ divides $(q'^2+1)(q'-1)$. Since now $m'\leq m$, we must have $m=m'$. Therefore $K/H\cong S$. Now $|G|=|K/H|=|S|$, and hence $G\cong S$.
\end{proof}



\def\cprime{$'$}

\end{document}